\documentclass[aoas,preprint]{imsart}
\RequirePackage[OT1]{fontenc}
\RequirePackage{amsthm,amsmath}
\RequirePackage{natbib}
\RequirePackage[colorlinks,citecolor=blue,urlcolor=blue]{hyperref}
\startlocaldefs
\numberwithin{equation}{section}
\theoremstyle{plain}

\endlocaldefs
\usepackage{amsthm}
\usepackage{mathtools}
\usepackage{longtable}
\usepackage{fancyhdr}
\usepackage{dcolumn}
\usepackage{comment}
\newtheorem{theorem}{Theorem}[section]
\newtheorem{lemma}[theorem]{Lemma}

\newtheorem{remark}{Remark}

\begin{document}

\begin{frontmatter}
\title{An Oracle Property of The Nadaraya-Watson Kernel Estimator for High Dimensional Nonparametric Regression}
\runtitle{K-Fold Cross-Validation and Kernel Regression}

\begin{aug}
\author{\fnms{Daniel} \snm{Conn}
\ead[label=e1]{djconn17@gmail.com}}
\and
\author{\fnms{Gang} \snm{Li}
\ead[label=e2]{vli@ucla.edu}}


\affiliation{University of California, Los Angeles
}

\address{Department of Biostatistics\\
University of California, Los Angeles\\
Los Angeles, California 90095\\
USA\\
\printead{e1}}

\address{Department of Biostatistics\\
University of California, Los Angeles\\
Los Angeles, California 90095\\
USA\\
\printead{e2}}

\end{aug}

\begin{abstract}
The celebrated Nadaraya-Watson kernel estimator is among the most studied method for nonparametric regression. A classical result is that its rate of convergence depends on the number of covariates and deteriorates quickly as the dimension grows, which underscores the ``curse of dimensionality'' and has limited its use in high dimensional settings. In this article,  we show that 
when the true regression function is single
or multi-index, the effects of the curse of dimensionality may be mitigated for the Nadaraya-Watson kernel estimator.  Specifically, we prove that with $K$-fold cross-validation,
the Nadaraya-Watson kernel estimator indexed by a positive semidefinite bandwidth matrix
has an oracle property that its rate of convergence depends
on the number of indices of the regression function rather than the number of covariates.  
Intuitively, this oracle property is a consequence of allowing the bandwidths to diverge to infinity as opposed to restricting them all to converge to zero at certain rates as done in previous theoretical studies.
Our result provides a theoretical perspective for the use of kernel estimation in high dimensional nonparametric regression and other applications such as metric learning when a low rank structure is anticipated.
Numerical illustrations are given through simulations and real data examples. 
\end{abstract}



\end{frontmatter}

%
%

\newcommand{\addlabel}[1]{%
\refstepcounter{equation}%
\label{#1}%
\let\]\endequation}
\newcommand{\resid}{(\hat{y}_{j} - y_{j})} 
\newcommand{\pdm[1]}{$\frac{\partial #1}{\partial H_{uv}}$}
\newcommand{\pd[1]}{\frac{\partial #1}{\partial H_{uv}}}
\newcommand{\pdh[1]}{\frac{\partial #1}{\partial h}}

\newcommand{\yhat}{\hat{y}_{j}}
\newcommand{\kij}{k_{ij}}
\newcommand{\kih}{k_{i,H}(x)}
\newcommand{\sumn}{\sum_{i=1}^{n}}
\newcommand{\phiH}{\psi_{H}(x|P_{n})}
\newcommand{\splitj}{S_{n}^{(j)}}
\newcommand{\maha}{(X_{i} - x)'H(X_{i} - x)}
\newcommand{\fold[1]}{\mathcal{F}_{#1}}
\newcommand\floor[1]{\lfloor#1\rfloor}
\newcommand\ceil[1]{\lceil#1\rceil}
\graphicspath{{./../simulations}}

%

\section{Introduction}
The Nadaraya-Watson kernel regression estimator \citep{nadaraya1964estimating, watson1964smooth} is a corner stone of 
nonparametric regression. Assume that one observes $n$ independent and identically (iid) distributed 
random variables of the form 
$O_{i}=(X_{i}, Y_{i}),  \sim P_{0}$, $i=1, \ldots,n$, where
$X_{i}\in \mathcal{R}^{p}$ are $p$-dimensional covariates and $Y_{i}$ is a real-valued outcome.  
The Nadaraya-Watson kernel estimator of the multivariate regression function $\psi(x)\equiv E[Y|X=x]$ is
commonly defined as
\begin{equation} 
\psi_{H^*}(x)=\frac{\sum_{i=1}^{n} K(H^{*-1/2}(X_i-x)) Y_{i}}{\sum_{i=1}^{n} K(H^{*-1/2}(X_i-x))}, \label{eq::kern_def0}
\end{equation}
where $H^*$ is a $p\times p$ symmetric positive definite matrix depending on $n$, 
and $K(u)$ is a kernel function such that $\int K(u) du =1$. 

It is well-known that the performance of nonparametric regression methods degrades as the
number of covariates, $p$, increases.  This degradation in performance is often referred to as the 
``curse of dimensionality."  For kernel regression, the curse of dimensionality can be seen to manifest
itself when $H^*=h^{*}_{n}I_{p}$ by considering results such as Theorem 5.2 of \cite{gyorfi2006distribution} 
concerning the rate of the convergence of the kernel regression estimator $\psi_{n,h^{*}_{n}}$ to $\psi$:
\begin{equation}
E\int (\psi_{n, h^{*}_n}(x) - \psi(x))^{2}dP_{0,X}(x) = O(n^{-2/(p + 2)}) \label{sclr_rate},
\end{equation} 
for an appropriately chosen sequence of scalar bandwidths, $h_{n}^{*}$.  
Also see the discussion at the end of Chapter 4 in \cite{hardle2012nonparametric}.  We highlight that 
the above result is established under the commonly made assumption that the bandwidth tends to zero at a certain rate as $n$ grows to infinity,
which allows the kernel estimator to pick up local features and balance the trade-off between its bias and variance.      

In this article,  however, we will show that if the bandwidths are allowed to 
diverge to infinity, then the Nadaraya-Watson kernel regression 
estimator using a bandwidth matrix has an oracle property when there exists an embedded low dimensional structure in $\psi(x)$. 
We will show in Theorem \ref{thm::kern_main} that if the true regression function is a single or multi-index regression model with index number $m$, 
then $K$-fold cross-validation can be used to produce an estimator with rate of convergence that depends on $m$ rather than $p$.
If $m$ is much less than $p$, the gain in predictive performance may be substantial.  

Specifically, we consider the following 
reparametrized form of 
the Nadaraya-Watson kernel estimator 
\begin{equation} 
\psi_{H}(x)=\frac{\sum_{i=1}^{n} K(H^{1/2}(X_i-x)) Y_{i}}{\sum_{i=1}^{n} K(H^{1/2}(X_i-x))}, \label{eq::kern_def}.
\end{equation}
Note that (\ref{eq::kern_def}) is equivalent to the classical definition (\ref{eq::kern_def0}) with $H^*=H^{-1}$ if $H$ is
positive definite, but we will allow $H$ to be positive semidefinite.  Letting $H$ be less than full rank, theoretically, allows
the above estimator to take advantage of low-dimensional structure in $\psi(x)$.
Our theoretical result on the oracle property of the Nadaraya-Watson kernel estimator relies on an extension of an oracle inequality concerning sample splitting and $K$-fold cross-validation presented
in \cite{dudoit2005asymptotics} and \cite{gyorfi2006distribution}, who considered sample splitting or cross-validation for selecting the best model from a discrete collection of models.
In practice, for kernel regression with a bandwidth matrix, specifying a discrete set of bandwidth matrices over which the $K$-fold 
cross-validation criterion is to be minimized is likely to be overly burdensome.  Our result allows for optimization
of the $K$-fold cross-validation criterion with respect to $H$ to take place over a bounded subset of the space of $p \times p$ positive semidefinite matrices.
This fact is important because it provides theoretical justification for the use of general optimization techniques such as gradient descent that rely
on selecting the optimal bandwidth matrix over a continuum of positive semidefinite matrices, rather than a discrete grid.  

The rest of the paper is organized as follows. 
In Section 2, we first discuss optimality criteria we use for assessing predictive performance of an estimator and
introduce notations for cross-validation schemes.  Then we present a general oracle inequality for estimators using $K$-fold 
cross-validation.  This result is then used to prove an oracle property for the kernel regression estimator defined in (\ref{eq::kern_def}).
In Sections 3 and 4, we assess the performance of our cross-validated kernel regression estimator in simulations and on commonly used benchmark 
data sets from the UC Irvine Machine Learning repository.   Our contributions are summarized in the discussion of Section 5.
Sections 6 contains an additional result concerning the convergence rate of a kernel regression estimator in which a Gaussian kernel is used and
recalls some facts concerning brackets and bracketing numbers.  Section 7 contains the proofs of Section 2.

\section{Main Results}
\subsection{Preliminaries}
\subsubsection{Optimality Criteria for Predictive Performance}
Given an estimator $\hat \psi$ of $\psi$ based on the observed data,
consider the squared error loss, $L(Y, \hat \psi(X))=(Y - \hat \psi(X))^{2}$, for an additional independent observation, $O=(X,Y) \sim P_{0}$. 
Define
\begin{equation}
\tilde{\theta}_{n}(\hat\psi) = E_{O}[L(Y, \hat\psi(X))] = \int L(y, \hat\psi(x))dP_{0}(x, y) \label{eq::crisk},
\end{equation}
to be the conditional risk \citep{keles2004asymptotically, dudoit2005asymptotics},  which is also referred to as the test error or generalization error \citep{friedman2001elements} or the integrated squared error ($ISE$) \citep{marron1986random}.

 Note that the conditional risk is a random variable as $\hat\psi$
depends on the observed data.  We call $E\tilde{\theta}_{n}(k)$ the marginal risk, or the expected test error, or
the mean  integrated squared error ($MISE$).  

Define $\psi(X)=E[Y|X]$ to be the true conditional expectation of $Y$ given $X$ and $\theta_{opt} = E_{O}[(Y - \psi(X))^{2}]$. Then, 
for any square integrable estimator $\hat\psi(X)$ ,
$$
\tilde{\theta}_{n}(\hat\psi) \geq \theta_{opt},
$$
 since $\psi(X)=E[Y|X]$ minimizes $E_{O}[(Y - \eta(X))^{2}]$ over all square integrable functions $\eta(X) $ of $X$ (see, e.g., Corollary 8.17 of \cite{klenke2013probability}).
Hence $\theta_{opt}$ is a lower bound for both the conditional risk and the marginal risk.
Moreover, we note that selecting the estimator $\hat\psi$ that minimizes the conditional risk is equivalent to minimizing the $\tilde{\theta}_{n}(\hat\psi) - \theta_{opt}$  and that  
$\tilde{\theta}_{n}(\hat\psi) - \theta_{opt} = \int ({\hat\psi}(x) - \psi(x))^{2}dP_{0,X}(x),$
where  $P_{0,X}$ is the marginal 
distribution of $X$.   We refer to $\tilde{\theta}_{n}(\hat\psi) - \theta_{opt}$ as the conditional excess risk and 
$E\tilde{\theta}_{n}(\hat\psi) - \theta_{opt}$ as the marginal excess risk.

Because the true distribution $P_0$ of the observations is unknown,  the conditional risk  must be estimated.  
A natural estimator of the conditional risk would be $\int L(y, \hat\psi(x))dP_{n}(x, y)$, where $P_{n}$ is the empirical distribution of the observations.
Unfortunately, this estimate may be highly optimistic because the data has been used twice to  
first produce the estimator, $\hat\psi$ and then obtain
the estimate of the conditional risk.  A sample splitting procedure, whereby a (training) subset of the data is set aside to produce $\hat\psi$
and a separate (validation) subset of the data is used to estimate $\tilde{\theta}_{n}(\hat\psi)$, would produce an estimate
of the conditional risk that is less prone to this negative bias.  Cross-validation schemes for estimating the conditional risk are elaborations of
the aforementioned sample splitting procedure, in which  observations alternate
 in their role of training and validation as described in the next section. 

 
\subsubsection{A Formal Explication of Cross-Validation}
Let $\{ \psi_k(x): k \in \Xi_{n} \}$ denote a collection of estimators of $\psi(x)$ indexed by $k$, where 
$\Xi_{n}\subset R^d$ for some positive integer $d$. For simplicity, we abuse the notation by using $k$ for $\psi_k$ from now on.
Below we formally present the concept of cross-validation, giving attention to the case in which the index $k$ runs through a continuous range of values.     
 
We denote a split of the data into training and validation sets via the binary vector
$S_{n} =(S_{n1},\ldots, S_{nn})^T \in \{0, 1\}^{n}$, where
\begin{equation*}
S_{ni} = \begin{cases}
0, & \text{if observation $i$ is in the training set}, \\
1, & \text{if observation $i$ is in the validation set},
\end{cases}
\end{equation*}
and the set $\{0, 1\}^{n}$ represents all possible splits of the data into training and validation sets. 
Define $P_{n,S_{n}}^{0}$ and $P_{n, S_{n}}^{1}$ as the empirical distribution of the observations in the training and validation set, respectively.
Let $\psi_{k}=\psi_{k}(X|P^{0}_{n, S_{n}})$ be an estimator produced by applying an estimation procedure to observations in the training set 
determined by $S_{n}$.  Define the conditional expectation,
given the observations in the training set, of a function $f(O, \psi_{k}(X|P^{0}_{n, S_{n}}))$ by $E[f(O, \psi_{k})|P^{0}_{n, S_{n}}]$,
where $f$ is a function depending on an independent observation $O \sim P_{0}$ and the estimator $\psi_{k}(X|P^{0}_{n, S_{n}})$  .

From now on, we will assume that all data splits devote the same proportion of observations, $1-\pi$, to
training. A cross-validation scheme is defined by assigning a set of $nsplit$ probability weights $w_{1},\ldots,w_{nsplit}$ such that
$w_{j} > 0$ and $\sum_{j=1}^{nsplit}w_{j}=1$ to a subset of $nsplit$ elements of $\{0, 1\}^{n}$.  The corresponding 
cross-validation criterion is then 
defined as  
\begin{flalign}
&\hat{\theta}_{n(1-\pi)}^{CV}(k) = E_{S_{n}}\int L(y, \psi_{k}(x|P^{0}_{n,S_{n}}))dP^{1}_{n,S_{n}}(x, y)\\ \label{eq::cvalcrit}
&\hspace{1.8cm}= \sum_{j=1}^{nsplit}w_{j}\int L(y, \psi_{k}(x|P^{0}_{n,S_{n}^{(j)}}))dP^{1}_{n,S_{n}^{(j)}}(x, y).
\end{flalign}

The $K$-fold cross-validation scheme is defined by splitting the $n$ observations into $K$ distinct subsets.  
This partition of the observations results in $K$ binary vectors $S_{n}^{(1)},\ldots, S_{n}^{(K)}$ where $S^{(j)}_{n}$ is
created by letting observations in the $j$th element of the partition serve as the validation set.  After minor modification,
we may take $\pi = \floor{n/K}/n$.  The $K$-fold cross-validation scheme then puts a probability weight of $w_j=1/K$ on each of these $K$ binary vectors.

A natural benchmark for a cross-validation scheme is
\begin{flalign}
&\tilde{\theta}^{CV}_{n(1-\pi)}(k)=E_{S_{n}}\int L(y, \psi_{k}(x|P^{0}_{n,S_{n}}))dP_{0}(x, y)\label{eq::cvbench}\\ 
&\hspace{1.8cm}=\sum_{j=1}^{nsplit}w_{j}\int L(y, \psi_{k}(x|P^{0}_{n,S_{n}^{(j)}}))dP_{0}(x, y),\label{eq::cvbenchsum}
\end{flalign}
which is referred to as the cross-validation benchmark, or the ``commensurate optimal benchmark"
\citep{dudoit2005asymptotics}). 
The cross-validation benchmark $\tilde{\theta}^{CV}_{n(1-\pi)}(k)$ 
can be regarded as a cross-validation criterion when an infinite
number of observations are available for validation.
Although minimization of the cross-validation benchmark is not equivalent to minimization of the conditional risk over
estimators that use all $n$ observations, for $K$-fold cross-validation, the following relationship holds:  $E\tilde{\theta}^{CV}_{n(1-\pi)}(k)=E\tilde{\theta}_{n(1-\pi)}(k)$, where $\tilde{\theta}_{n(1-\pi)}(k)$ denotes the conditional risk of $\psi_k$. 
Thus, the cross-validation benchmark has mean equal to the marginal risk based on approximately $n(1 - \pi)$ observations rather
than $n$ observations.  

If minimization of the cross-validation criterion is carried out over a continuous range $\Xi_n$ of 
$\mathcal{R}^{d}$, 
there may not exist a $\hat{k}_{n} \in \Xi_{n}$ such that $\hat{\theta}_{n(1-\pi)}^{CV}(\hat{k}_{n}) = \inf_{k \in \Xi_{n}}\hat{\theta}_{n(1-\pi)}^{CV}(k)$.
For this reason, 
we will consider a $\gamma$-suboptimal point  $\hat{k}_{n} \in \Xi_{n}$  as in \cite{boyd2004convex} 
such that 
$\hat{\theta}_{n(1-\pi)}^{CV}(\hat{k}_{n}) \leq \text{inf}\{\hat{\theta}_{n(1-\pi)}^{CV}(k): k\in \Xi_{n}\}  + \gamma$, for some pre-specified $\gamma > 0$.
Note that a $\gamma$-suboptimal point is not necessarily unique. If a minimizer exists in $\Xi_{n}$, then it can be considered as a $\gamma$-suboptimal  point with $\gamma=0$.  

\subsection{An Oracle Property of the Kernel Regression Estimator}

Our main theoretical result relies on an extension of an oracle inequality presented
in \cite{dudoit2005asymptotics} and \cite{gyorfi2006distribution}, who considered sample splitting or cross-validation for selecting the best model from a discrete collection of models.
Our extension below allows for optimization
of the $K$-fold cross-validation criterion to take place over a continuous bounded subset of the space of $p \times p$ positive semidefinite matrices and consequently enables selecting the best model from a continuum of models.

\subsubsection{An Oracle Inequality for  $K$-Fold Cross Validation}
We now present an oracle inequality that demonstrates that $K$-fold cross-validation produces an estimator that
has near optimal performance according to the cross-validation benchmark.   
Let $\hat{k}_{n}$ and $\tilde{k}_{n}$ be the $\gamma$-suboptimal minimizers
of the $K$-fold cross-validation criterion $\hat{\theta}_{n(1-\pi)}^{CV}(k)$ and its cross-validation benchmark $\tilde{\theta}^{CV}_{n(1-\pi)}(k)$, respectively, over $\Xi_n$.  

The following assumptions will be needed.
\begin{itemize}

\item[(A1)] There exists a constant $M>0$ such that $Pr(|Y| < M) = 1$ and $sup_{X}|\psi_{k}(X|P_{n})| \leq M < \infty$ almost surely for all $k \in \Xi_n$, where $M$ does not depend on $n$ and the supremum is over the support of the marginal distribution of $X$.  
\item[(A2)]  $\Xi_{n}$ is a bounded set.  
\item[(A3)]  Assume that with probability 1, $L(Y, \psi_{k}(X|P_{n})) - L(Y, \psi(X))$, as a function of $k$, is Lipshitz continuous
with Lipshitz constant $C$.  This Lipshitz constant, $C$, does not depend on the training set, $P_{n}$.
\item[(A4)]  Let $c_{1}(n\pi, d, \Xi_{n}, M, \delta)=  \log\Big\{\Big(\frac{n\pi}{c_{2}(M, \delta)}\Big)^{d}4(4\sqrt{d}C^{2}4(1+ 2\delta)\text{diam}(\Xi_{n}))^{d}\Big\}$ and
let $c_{2}(M, \delta) = 8(1 + \delta)^{2}(\frac{M_{2}}{\delta} + \frac{M_{1}}{3})$, where $M_{1}=8M^{2}$ and $M_{2}=16M^{2}$. Assume $\delta > 0$ is such that $1/c_{2}(M, \delta) \leq M_{1}8(\frac{1}{(1 + 2\delta)}+\frac{1}{3})$ and take $n$ large enough that
$c_{1}(n\pi, d, \Xi_{n}, M, \delta) > 1$. 
\end{itemize}
The derivation of our result in the theorem below requires that terms of the form $L(Y, \psi_{k}(X|P_{n})) - L(Y, \psi(X))$, where $P_{n}$ depends on the split,
converge to their expectation uniformly over $k \in \Xi_{n}$.  Intuitively, if $L(Y, \psi_{k}(X|P_{n})) - L(Y, \psi(X))$ converges
to its expectation uniformly in $k$, then the $K$-fold cross-validation criterion will converge to the $K$-fold cross-validation benchmark uniformly in $k$. 
Assumptions (A1), (A2), and (A3) ensure this occurs.  Assumption (A4) is a technical condition primarily serving to simplify the result.  

\begin{theorem}
\label{thm::main_extn}
Assume the above assumptions (A1)-(A4) hold.
Then, 
\begin{flalign}
&0 \leq E\tilde{\theta}^{CV}_{n(1-\pi)}(\hat{k}_{n}) - \theta_{opt} \leq (1+2\delta)(E\tilde{\theta}^{CV}_{n(1-\pi)}(\tilde{k}_{n}) - \theta_{opt})  + \label{eq::oracle_err}&\\
&\-\hspace{4.5cm} \frac{4c_{2}(M,\delta)}{n\pi}c_{1}(n\pi, d, \Xi_{n}, M, \delta)+\label{eq::cv_err}&\\
&\-\hspace{4.5cm} (1 + \delta)\gamma.                                       
\end{flalign}
\end{theorem}
\begin{remark}
 Our finite sample result in Theorem \ref{thm::main_extn} suggests a trade-off with regard to the 
 choice of what proportion of observations to use for validation versus training for the upper-bound of the above inequality. 
 The expression on the right hand side of the above inequality depends heavily on the number of observations 
 used for validation: $n\pi$.   If $\pi$ is too small, the number of observations used for validation will be small
 and the term in (\ref{eq::cv_err}) will be large, which signals difficulty in identifying the model
 that minimizes the cross-validation benchmark.  If $\pi$ is large, the expectation of the optimal 
 cross-validation benchmark based on training sets of size $n(1-\pi)$
 may be a poor approximation for the expectation of the optimal conditional risk based on a training set of size $n$.  
A similar trade-off has been observed with regard to the choice of $K$ in \citep{kohavi1995study,friedman2001elements,james2013introduction}.
\cite{arlot2010survey}, perhaps providing a more refined discussion of this issue, suggest that the
choice of $K$ should depend, in part, on the signal-to-noise ratio.
\end{remark}

\begin{remark}
Because $E\tilde{\theta}^{CV}_{n(1-\pi)}(\tilde{k}_{n}) - \theta_{opt})$ is less than $E\tilde{\theta}^{CV}_{n(1-\pi)}(\hat{k}_{n}) - \theta_{opt}$, the
above inequality states that on average cross-validation selects an estimator with close to optimal performance if $n\pi$ large, where, once again,
performance is measured by the cross-validation benchmark.
Note that the tightness of the above inequality depends on $n$, $\pi$, $d$, and $diam(\Xi_{n})$. 
The bound becomes looser for larger values of $d$ and $diam(\Xi_{n})$.  As $n\pi$ grows,
the bound becomes tighter.
\end{remark}

The above oracle inequality in Theorem \ref{thm::main_extn} is derived for any collection of estimators indexed by $k$
that ranges over a continuous bounded subset $\Xi_n$ of $R^d$, {satisfying (A1)-(A4)}. In the next section, we apply Theorem \ref{thm::main_extn} to the Nadaraya-Watsion kernel regression estimator defined in (\ref{eq::kern_def}) with a Gaussian kernel $K(u)= \exp(-||u||^{2})$, indexed by a bandwidth matrix $H$.  
We also show that for this Nadaraya-Watson kernel regression estimator,  a minimizer  (corresponding to $\gamma=0$) exists for
both the $K$ fold cross-validation criterion and the $K$-fold cross-validation benchmark.   

\subsubsection{An Oracle Property for a Nadaraya-Watson Estimator with a Matrix Valued Bandwidth}
As in \cite{boyd2004convex}, we represent the set of symmetric positive semidefinite matrices, $S^{p}_{+}$, as elements
in $\mathcal{R}^{p(p+1)/2}$.  Explicitly, we represent a particular matrix $H$ as $(h_{1,1},h_{2,2},\ldots,h_{p,p},\\h_{2,1},h_{3,1},\ldots,h_{p,p-1})'$. 
The Frobenius norm of a symmetric matrix, $H$, viewed as an element in $\mathcal{R}^{p(p+1)/2}$, 
is defined as $||H||_{F}=(\sum_{i=1}^{p}h_{ii}^{2} + 2\sum_{j < i}h_{ij}^{2})^{1/2}$.  
To ensure that $\Xi_{n}$ is bounded (A2), we will let $\Xi_{n}$ 
consist of the elements of $S^{p}_{+} \subset  \mathcal{R}^{p(p+1)/2}$ with Frobenius norm less than or equal to $\lambda_{n}$
for some  $\lambda_{n}>0$.

\begin{theorem}
\label{thm::kern_main}
Consider the class of kernel regression estimators $\psi_{H}$, defined in (\ref{eq::kern_def}), where $K(u)=\exp(-||u||^{2})$ is the Gausian kernel function and 
$H$ is selected via $K$-fold cross-validation.  Assume that $P(|Y| < M)=1$ and that $(A4)$ holds.  
Assume  further that there exists a constant $B > 0$ such that $P(||X|| < B)=1$.  
Then, we have the following finite sample result. 
\begin{itemize}
\item[(a)] Minimizers of the $K$-fold cross-validation criterion and
the $K$-fold cross-validation benchmark with respect to $H$ exist in $\Xi_{n}$.  

\item[(b)]
Denote by $\hat{H}_{n} \in \Xi_{n}$ a minimizer of the $K$-fold cross-validation criterion and
$\tilde{H}_{n} \in \Xi_{n}$ a minimizer of the $K$-fold cross-validation benchmark. Then,
the assumptions $(A1)$, $(A2)$, and $(A3)$ are 
satisfied and consequently the inequality of Theorem \ref{thm::main_extn} holds, with $\gamma=0$, $d=p(p + 1)/2$
and $C=64\sqrt{p(p+1)/2}B^{2}M^{2}$. 

\item[(c)]
Let $\psi_{n(1-\pi), \hat{H}_{n}}(x|P^{0}_{n, S_{n}^{(j)}})$ be the kernel regression estimator obtained by using $\hat{H}_{n}$ from part (b) 
and the $j$th training sample for any $1 \leq j \leq K$.
Assume $\psi(X)=\phi(Tx)$ is a multi-index model defined such that 
$\phi:\mathcal{R}^{m}\rightarrow \mathcal{R}$ is Lipshitz continuous with Lipshitz constant $R$ and $T$ is an $m\times p$ orthogonal matrix.
Let $\lambda_{n} = \sqrt{p}V(\log(n)n)^{1/3}$, where $V > 0$ is a positive constant.
Then, as $n\to \infty$,
\begin{flalign}
&\-\hspace{3.975cm}E\tilde{\theta}_{n(1-\pi)}^{CV}(\hat{H}_{n}) - \theta_{opt}=&\\
&E\int (\psi_{n(1-\pi),\hat{H}_{n}}(x|P^{0}_{n, S_{n}^{(j)}}) - \psi(x))^{2}dP_{0,X}(x) = O(\log(n(1 - \pi))^{\frac{m}{m + 2}}(n(1 - \pi))^{\frac{-2}{m + 2}}).
\end{flalign}
\end{itemize}

\end{theorem}  
\begin{remark}
The result of part (c) is an analogue to the result (\ref{sclr_rate}).
In contrast to the result (\ref{sclr_rate}), the rate of convergence of the above estimator depends on $m \leq p$.
\end{remark}


\begin{remark}
If 
\begin{equation}
                \frac{c_{1}(n\pi, d, \Xi_{n}, M, \delta)}{(n\pi)\{E\tilde{\theta}^{CV}_{n(1-\pi)}(\tilde{k}) - \theta_{opt} \}}  \rightarrow 0 \text{  as $n$} \rightarrow \infty,
\end{equation}
then it follows, by dividing both sides of the oracle inequality in Theorem 2.1, that 
\begin{equation}
 \frac{E\tilde{\theta}^{CV}_{n(1-\pi)}(\hat{k}) - \theta_{opt}}{E\tilde{\theta}^{CV}_{n(1-\pi)}(\tilde{k}) - \theta_{opt}}  \rightarrow 1  \text{  as $n$} \rightarrow \infty.  \label{eq::erat}
 \end{equation}
This implies that the estimator produced by $K$-fold cross-validation estimator has asymptotically optimal predictive performance as measured
by cross-validation benchmark.

Assuming that $\lambda_{n}$ grows at the rate specified in Theorem \ref{thm::kern_main},
the above condition will then be satisfied if $n\pi(E\tilde{\theta}^{CV}_{n(1-\pi)}(\tilde{H}_{n}) - \theta_{opt})$ increases at a
polynomial rate, $n^{\gamma}$, where $0 < \gamma < 1$.  As $n^{-1}$ is a ``parametric rate" of convergence we would expect 
$E\tilde{\theta}^{CV}_{n(1-\pi)}(\tilde{H}_{n}) - \theta_{opt}$ to decrease at a rate slower than $n^{-1}$.  Thus, we expect 
the condition to hold in many circumstances.  In fact, Remark 4 of \cite{antos2000lower} provides an example 
of a class of distribution within which distributions can be found such that it holds.
\end{remark}

\begin{remark}
The $K$-fold cross-validation criterion is differentiable with respect to $H$ for the Nadaraya-Watson estimator
we consider.  The derivative is presented in the Supplementary material.  To find $\hat{H}_{n}$, we have 
implemented a variant of the gradient descent algorithm presented in \cite{weinberger2007metric}.  Our simulation
results and the data analysis demonstrate that this algorithm is capable of finding acceptable local minimizers.
Development of more sophisticated algorithms for finding global minimizers as well as improvement of computational
speed warrants further research.     
\end{remark}

%

\section{Simulations}
In this section, we present the results of a simulation study to
illustrate the performance of the kernel regression estimator using a matrix
valued bandwidth when the true regression model is single or two-index.
We also compare it with a kernel regression estimator
using a scalar valued bandwidth and an oracle kernel regression estimator that 
is derived by calculating the true indices and carrying out kernel
regression using these indices as covariates.
We use 10-fold cross-validation to estimate the optimal scalar-valued and matrix-valued bandwidth parameter.
To find the matrix-valued bandwidth, we applied gradient descent algorithm
to minimize the 10-fold cross-validation criterion.  A grid search was used to find the scalar-valued bandwidth.  
For the oracle estimator, a grid search was used to determine its optimal smoothing parameter, $\tilde{h}_{n}$,
with the goal of minimizing the conditional risk $\tilde{\theta}_{n}(h)$, where
$\tilde{h}_{n}$ is scalar for the single-index model and a 2-dimensional vector for the 2-index model.
We also demonstrate how the performance of kernel regression, parametrized by a bandwidth matrix, degrades
as the number of covariates increases.  


In all simulations, the covariates are independent and standard normal random variables.   
For the single index regression model we set
\begin{equation}
Y = 2X_{1} + 2X_{2} + X_{3} + X_{4} + \epsilon \label{eq::sir_model},
\end{equation}
where $\epsilon \sim N(0, \sigma^{2})$.
The standard deviation of the error term, $\sigma$, has been set to $0.15$ and $0.3$. 
Additional ``noise" covariates (covariates with coefficient equal to 0) have been added such that
$p$ varies from 5, 10, to 20.
For the two-index regression model, we set 
\begin{equation}
Y = T_{1}\sin(\sqrt{5}T_{2}) + T_{2}\sin(\sqrt{5}T_{1}) + \epsilon, \label{eq::dir_model}
\end{equation}
which was explored previously in \cite{polzehl2009note},
where the two indices are defined as
$T_{1}=X_{1}/\sqrt{5}  + 2X_{2}/\sqrt{5}$ and $T_{2}=-2X_{1}/3  + X_{2}/3 + 2X_{3}/3$,
and $\epsilon \sim N(0, \sigma^{2})$.  Once again, additional noise covariates have been added such that $p$ varies from 5, 10, to 20
and $\sigma$ takes the values 0.15 and 0.3.

The performance of these estimators is measured by taking the square root of the mean squared error (RMSE)
on a test set of size 10,000.  Each simulation scenario was repeated 100 times.  The mean of the RMSEs over the 100 simulations 
as well as 95\% confidence intervals for each simulation scenario are presented in Table \ref{tab::sim_results}.

The kernel regression estimator using the matrix-valued bandwidth 
outperforms the kernel regression estimator using a scalar-valued bandwidth over all simulation
scenarios.  The rate of convergence  of the kernel regression estimator using a scalar-valued bandwidth 
is slow enough that the RMSE changes by very small amounts as the sample size increases.  We found that the
scalar bandwidth selected by 10-fold cross-validation led to significant over-smoothing for all
simulation scenarios.  

%
%
%

The performance of the kernel regression estimator using a matrix-valued bandwidth does indeed degrade as
additional noise covariates are included, however, its performance is substantially improved as the sample size
increases.  For a sample size of 1,000, even when $p=20$, this estimator is comparable to that of the oracle kernel regression
estimator.  

\begin{table}
\centering
\begin{tabular}{ccccccc}
  \hline
$n$ & $p$ & SD & Index \# & Scalar Bandwidth & Matrix Bandwidth & Oracle Estimator \\ 
  \hline
250 & 5 & 0.15 & 1 & 1.835 (1.833,1.838) & 0.165 (0.164,0.167) & 0.164 (0.162,0.165) \\ 
  250 & 10 & 0.15 & 1 & 1.837 (1.834,1.839) & 0.171 (0.169,0.173) & 0.167 (0.165,0.169) \\ 
  250 & 20 & 0.15 & 1 & 1.835 (1.833,1.837) & 0.175 (0.173,0.177) & 0.166 (0.164,0.168) \\ 
  500 & 5 & 0.15 & 1 & 1.834 (1.831,1.836) & 0.158 (0.157,0.158) & 0.157 (0.156,0.158) \\ 
  500 & 10 & 0.15 & 1 & 1.834 (1.831,1.836) & 0.158 (0.158,0.159) & 0.157 (0.156,0.158) \\ 
  500 & 20 & 0.15 & 1 & 1.834 (1.831,1.836) & 0.161 (0.16,0.162) & 0.157 (0.157,0.158) \\ 
  1000 & 5 & 0.15 & 1 & 1.832 (1.83,1.835) & 0.154 (0.154,0.154) & 0.153 (0.153,0.154) \\ 
  1000 & 10 & 0.15 & 1 & 1.832 (1.829,1.834) & 0.154 (0.154,0.155) & 0.153 (0.153,0.154) \\ 
  1000 & 20 & 0.15 & 1 & 1.833 (1.83,1.835) & 0.155 (0.154,0.155) & 0.153 (0.153,0.153) \\ 
  250 & 5 & 0.30 & 1 & 1.852 (1.849,1.854) & 0.318 (0.316,0.319) & 0.313 (0.312,0.315) \\ 
  250 & 10 & 0.30 & 1 & 1.853 (1.851,1.855) & 0.325 (0.323,0.327) & 0.314 (0.312,0.315) \\ 
  250 & 20 & 0.30 & 1 & 1.854 (1.852,1.857) & 0.369 (0.363,0.376) & 0.313 (0.312,0.314) \\ 
  500 & 5 & 0.30 & 1 & 1.852 (1.85,1.855) & 0.309 (0.309,0.31) & 0.307 (0.307,0.308) \\ 
  500 & 10 & 0.30 & 1 & 1.851 (1.848,1.853) & 0.311 (0.311,0.312) & 0.307 (0.306,0.307) \\ 
  500 & 20 & 0.30 & 1 & 1.854 (1.852,1.856) & 0.338 (0.331,0.345) & 0.307 (0.307,0.308) \\ 
  1000 & 5 & 0.30 & 1 & 1.852 (1.849,1.854) & 0.305 (0.304,0.305) & 0.304 (0.303,0.304) \\ 
  1000 & 10 & 0.30 & 1 & 1.852 (1.85,1.855) & 0.306 (0.305,0.306) & 0.304 (0.303,0.304) \\ 
  1000 & 20 & 0.30 & 1 & 1.852 (1.849,1.854) & 0.312 (0.309,0.314) & 0.304 (0.303,0.304) \\ \hline
  250 & 5 & 0.15 & 2 & 0.787 (0.786,0.788) & 0.189 (0.188,0.19) & 0.187 (0.186,0.188) \\ 
  250 & 10 & 0.15 & 2 & 0.789 (0.788,0.79) & 0.196 (0.195,0.197) & 0.187 (0.186,0.188) \\ 
  250 & 20 & 0.15 & 2 & 0.788 (0.787,0.789) & 0.228 (0.222,0.234) & 0.187 (0.186,0.189) \\ 
  500 & 5 & 0.15 & 2 & 0.788 (0.786,0.789) & 0.172 (0.172,0.173) & 0.172 (0.172,0.173) \\ 
  500 & 10 & 0.15 & 2 & 0.788 (0.787,0.789) & 0.174 (0.174,0.175) & 0.172 (0.172,0.173) \\ 
  500 & 20 & 0.15 & 2 & 0.788 (0.787,0.789) & 0.183 (0.182,0.184) & 0.172 (0.172,0.173) \\ 
  1000 & 5 & 0.15 & 2 & 0.788 (0.787,0.789) & 0.163 (0.162,0.163) & 0.162 (0.162,0.163) \\ 
  1000 & 10 & 0.15 & 2 & 0.788 (0.787,0.789) & 0.164 (0.163,0.164) & 0.162 (0.162,0.163) \\ 
  1000 & 20 & 0.15 & 2 & 0.786 (0.785,0.787) & 0.166 (0.166,0.167) & 0.162 (0.162,0.162) \\ 
  250 & 5 & 0.30 & 2 & 0.831 (0.83,0.832) & 0.339 (0.338,0.341) & 0.333 (0.332,0.334) \\ 
  250 & 10 & 0.30 & 2 & 0.831 (0.829,0.832) & 0.356 (0.353,0.358) & 0.332 (0.331,0.334) \\ 
  250 & 20 & 0.30 & 2 & 0.831 (0.829,0.832) & 0.404 (0.401,0.408) & 0.333 (0.332,0.335) \\ 
  500 & 5 & 0.30 & 2 & 0.83 (0.828,0.831) & 0.323 (0.322,0.324) & 0.32 (0.319,0.321) \\ 
  500 & 10 & 0.30 & 2 & 0.829 (0.828,0.83) & 0.329 (0.328,0.329) & 0.319 (0.318,0.32) \\ 
  500 & 20 & 0.30 & 2 & 0.83 (0.829,0.831) & 0.359 (0.355,0.362) & 0.32 (0.319,0.32) \\ 
  1000 & 5 & 0.30 & 2 & 0.829 (0.828,0.83) & 0.313 (0.312,0.313) & 0.312 (0.312,0.313) \\ 
  1000 & 10 & 0.30 & 2 & 0.829 (0.828,0.83) & 0.315 (0.315,0.316) & 0.312 (0.311,0.313) \\ 
  1000 & 20 & 0.30 & 2 & 0.829 (0.828,0.83) & 0.324 (0.323,0.325) & 0.312 (0.312,0.313) \\ 
   \hline
\end{tabular}
\caption{RMSEs and 95\% CIs for kerned regression estimator with scalar bandwidth, bandwidth matrix and oracle kernel regression estimator under a single index model (Index \# =1) and a two-index model (Index \# =2).}
\label{tab::sim_results}
\end{table}

\section{Data Analysis}
In this section we demonstrate the estimated predictive performance of kernel regression using a bandwidth matrix versus kernel regression using a scalar bandwidth
with the goal of understanding whether the asymptotic results presented in this article are indicative of finite sample performance on
commonly explored data sets.  For reference, we also compare the predictive performance of these kernel regression estimators to that of
a linear regression estimator with no interactions and no nonlinear terms.  Large differences
in estimated prediction accuracy between the nonparametric kernel regression methods and the linear model may provide an indication of 
whether the linear model is failing to account for any interactions or nonlinearities although this comparison does not constitute a formal test.  

To estimate the predictive accuracy of each method we used testing sets of approximate size $n\times0.25$.  Observations were split
into 4 groups.  This yielded 4 splits of the data into a set of size $n\times0.75$ for training and a set of size $n\times0.25$ for
estimating the RMSE via a testing set.  A final estimate of the RMSE was obtained by averaging the 4 estimates of the RMSE associated
with each split.  Within each training set 10-fold cross-validation was used to select the bandwidth matrix and scalar bandwidth for the two regression 
estimators.           

We tested these methods on 3 data sets: the Boston housing data set obtained via the R package MASS and
the concrete compressive strength, and the auto-mpg data sets, with the latter two being available from the UC Irvine Machine Learning 
Repository (https://archive.ics.uci.edu/ml).  Continuous covariate were centered and scaled to have variance equal to 1.
Discrete variables were coded using the usual dummy coding, thus, a discrete valued variable with $v$ levels contributes
$v-1$ covariates.  The results are presented in Table \ref{tab::data_RMSEs}.  

The kernel regression estimator indexed by a bandwidth matrix yielded a smaller estimated RMSE than both the kernel regression
estimator indexed by a scalar and the linear regression estimator.  The kernel regression estimator indexed by a scalar bandwidth
had the highest RMSE for all 3 data sets.  10-fold cross-validation led to over-smoothed estimates of the regression function. 
The improvement in performance from using a bandwidth matrix over a scalar bandwidth suggests that there is lower-dimensional
structure which the bandwidth matrix is able to take advantage of.  The kernel regression estimator indexed by a bandwidth matrix
may be performing better than linear regression because there is nonlinear structure or there are interactions that the linear 
regression estimator fail to account for.  
\begin{table}
\centering
\begin{tabular}{cccccc}
\hline
Data Set & $n$ & $p$ & Bandwidth Matrix & Scalar Bandwidth & Linear Regression\\ 
\hline
Boston housing & 506 & 13 & 4.05 & 9.04 & 4.87\\
concrete & 1030 & 8 & 6.29 & 16.67 & 10.47\\ 
auto-mpg & 398 & 11 & 2.69 & 7.72 & 3.32\\ 
\hline
\end{tabular}
\caption{Estimated RMSEs for three commonly used data sets}
\label{tab::data_RMSEs}
\end{table}

\section{Discussion}
To the best of our knowledge, the performance of $K$-fold cross-validation for Nadaraya-Watson kernel regression estimator, indexed by a bandwidth matrix
has not previously been investigated.  We proved a finite sample oracle inequality for this estimator and demonstrated that lower-dimensional rates of convergence
can be achieved if the true regression model is single or multi-index.
The performance of cross-validation for the Nadaraya-Watson estimator has been investigated extensively in the literature
\citep{wong1983consistency, hall1984asymptotic, hardle1985optimal, hardle1987nonparametric, walk2002cross, gyorfi2006distribution}.
The results on cross-validation and kernel regression presented in \cite{hall1984asymptotic, hardle1985optimal, hardle1987nonparametric}, 
and \citet{walk2002cross} concern leave-one-out cross-validation and all consider the case when $H^{*}=h^{*}I_{p}$, where $h^{*}=1/h$.   
The results of these papers are distinguished by the optimality criterion they consider, assumptions about the distribution of the data, properties of the
kernel, and assumptions about the range over which the bandwidth is varied, which are detailed below. 

Our main result provides an upper-bound on the rate of convergence of the marginal excess risk.  Compared to the 
results in \cite{hall1984asymptotic, hardle1985optimal, hardle1987nonparametric}, and \cite{walk2002cross} we make the fairly strong assumption that 
both $Y$ and $X$ are almost surely bounded, although we do not assume that $p_{0,X}(x)$ is bounded as in 
\cite{hardle1985optimal} and \cite{hardle1987nonparametric}.  We also do not assume that $p_{0,X}(x)$ satisfies
a smoothness condition such as H\"older continuity, as in \cite{hardle1987nonparametric}.  \cite{hardle1987nonparametric} assume that 
$p_{0,X}(x)$ is known.  \cite{hall1984asymptotic, hardle1985optimal, hardle1987nonparametric} and \cite{walk2002cross}
assume the kernel has bounded support.   In contrast, we assume the use of the Gaussian kernel, a kernel with unbounded support.  
Unlike \cite{hall1984asymptotic} and \cite{hardle1985optimal}, we do not require that the bandwidth parameters converge to 0.  
\cite{walk2002cross} does not require that the scalar bandwidth converge 0, however, the result does assume the bandwidth parameter
varies over a discrete grid.  Our result allows the bandwidth matrix to range over a bounded subset of the space of positive semidefinite matrices and 
these subsets grow at an appropriately fast rate.       

The results presented in \cite{dudoit2005asymptotics} and \cite{gyorfi2006distribution} apply to a wide variety of estimation procedures.
The results of  \cite{hall1984asymptotic, hardle1985optimal, hardle1987nonparametric}, and  \cite{walk2002cross} are specific to kernel regression.
Our extension, as in \cite{dudoit2005asymptotics} and \cite{gyorfi2006distribution}, can be applied to regression estimators other than 
kernel regression.  However this generality has theoretical drawbacks.  Our proof requires that the number of folds, $K$, be constant
or grow slowly as a fraction of the sample size.  In particular, this excludes leave-one-out cross-validation.    
The upper-bound from Theorem \ref{thm::kern_main} may be loose and the rate of convergence may be substantially smaller.  
In particular, the result does not imply that leave-one-out cross-validation is inconsistent for the Nadaraya-Watson estimator, 
indexed by a matrix bandwidth.         

While the focus of the article is on $K$-fold cross-validation.  Our results easily generalize to other cross-validation
schemes.  For example, a cross-validation scheme that may lead to more stable model selection
could be obtained by specifying more than just $K$ splits in which approximately $n\pi$ observations are
used for validation.  Repeated $K$-fold cross-validation (\cite{burman1989comparative}),
wherein $K$-fold cross-validation is carried out repeatedly by using a different partition of the observations
each time, is such a scheme.  Our results also hold for repeated $K$-fold cross-validation.

It is of interest to investigate other settings in which Theorem \ref{thm::main_extn} can be applied.
For example, consider the ridge regression estimator defined as $\psi_{\lambda}(X|P_{n})=
 \text{argmin}_{\beta}\{\sum_{i=1}^{n}(Y_{i} - X'_{i}\beta)^{2} + \lambda \sum_{j=1}^{p}\beta_{j}^{2}\}$ (\cite{hoerl1970ridge}).
An alternate estimator is defined by penalizing each covariate to a different degree:
$\psi_{\lambda}(X|P_{n})=\text{argmin}_{\beta}\{\sum_{i=1}^{n}(Y_{i} - X'_{i}\beta)^{2} + \sum_{j=1}^{p}\lambda_{j}\beta_{j}^{2}\}$.
For this estimator, covariates with no effect or a weak effect on $Y$ should be penalized more heavily.
For both estimators, the $K$-fold cross-validation criterion may also be minimized via a gradient descent algorithm.
It may be possible to apply Theorem \ref{thm::main_extn} to both estimators, allowing for some modification of these 
ridge regression estimators to ensure that the estimators are bounded.

\section{Proofs} 
\subsection{Rate of Convergence of a Kernel Regression Estimator with Unbounded Support}
In this section, we present a result extending Theorem 5.2 of \cite{gyorfi2006distribution} on the rate of convergence of the kernel regression estimator
to handle the case where a Gaussian kernel is used.  This result is used in the proof of Theorem \ref{thm::kern_main}.

In the following lemma, we restrict our attention to the special case where a single bandwidth parameter, $h > 0$, is used.  Therefore,
instead of $k_{i,H}(x)= \exp(-\maha)$, we have $k_{i,h}(x)=\exp(-h(X_{i} - x)'(X_{i} - x))=\exp(-h||X_{i} - x||^{2}_{2})$, where 
$||X_{i} - x||^{2}$ is the squared Euclidean distance between $X_{i}$ and $x$.  With $h$ fixed, we denote the resulting kernel estimator
as $\psi_{n}$.  

\cite{gyorfi2006distribution} consider a class of kernel regression estimators of the form
\begin{equation}
\psi_{n}(x) = \frac{\sum_{i=1}^{n}K((X_{i}- x)h)Y_{i}}{\sum_{i=1}^{n}K((X_{i}- x)h)},
\end{equation}
where $K:\mathcal{R}^{d} \rightarrow [0, \infty)$ is the kernel function.
Theorem 5.2 of \cite{gyorfi2006distribution} provides an upper-bound on $E_{O}[(\psi_{n}(X) - \psi(X))^{2}]$ when $K(x)$ is a type of 
kernel with bounded support called a ``boxed" kernel.  Letting $S_{x,r}$ be a ball of radius $r$, centered at $x$, $K(x)$ is a boxed
kernel if there exists $0 < r < r'$ and $b > 0$ such that $I_{\{x \in S_{0,r'}\}} \geq K(x) \geq bI_{\{x \in S_{0,r}\}}$.    In the case of a 
Gaussian kernel, there does exist $b$ and $r>0$ such that $K(x) \geq bI_{\{x \in S_{0,r}\}}$.  However, the Gaussian kernel has unbounded
support, therefore, there exists no $r'$ such that $I_{\{x \in S_{0,r'}\}} \geq K(x)$.  The proof of Lemma \ref{lemma::kern_lemma} is given 
in the supplementary materials \cite{connSupplement}.

\begin{lemma}
\label{lemma::kern_lemma}
Assume $X$ has support such that there exists $B > 0$ such that $P(||X || < B) = 1$.
Assume the true regression function, $\psi(x)$ is Lipshitz continuous over the support of $X$, with Lipshitz constant $R$,
and that $\text{Var}(Y|X=x) \leq \sigma^{2}$ over the support of $X$.
Under these assumptions we have
\begin{equation}
E[(\psi_{n}(X) - \psi(X))^{2}] \leq  \frac{R^{2}\log(n)}{h} + \frac{2R^{2}B^{2}\tilde{c}h^{p/2}}{nb} + \frac{2\sigma^{2}\tilde{c}h^{p/2}}{nb} \label{inq::kernel_bound1},
\end{equation}
where $\tilde{c}$ depends on $p$ and $B$.  Furthermore, if we let the bandwidth $h$ increase to infinity as
\begin{equation*}
h^{*}_{n} = \bigg(\frac{A_{1}\log(n)n}{A_{2}\frac{p}{2}}\bigg)^{\frac{2}{p + 2}}
\end{equation*}
where $A_{1}=R^{2}$ and $A_{2}=\frac{2\tilde{c}}{b}(R^{2}B^{2} + \sigma^{2})$,
then the above bound yields
\begin{equation*}
E[(\psi_{n}(X) - \psi(X))^{2}] \leq A\log(n)^{\frac{p}{p+2}}n^{\frac{-2}{p + 2}},
\end{equation*}
where $A = A_{1}^{\frac{p}{p+2}}A_{2}^{\frac{2}{p+2}}\big((p/2)^{\frac{2}{p+2}} + (p/2)^{\frac{p}{p+2}}\big).$
\end{lemma}

\subsection{Additional Definitions and Results Regarding Brackets}
We briefly introduce the concept of brackets and bracketing numbers.  Let $\mathcal{F}$ be a collection of functions $f(O)$.
An $L_{ 2}$ $\epsilon$-bracket determined by a pair of functions, $l$ and $u$, such that $l(O) \leq u(O)$ and $E[|l(O) - u(O)|^{2}] \leq \epsilon^{2}$, is
defined as the set of functions $f(O) \in \mathcal{F}$ such that $l(O) \leq f(O) \leq u(O)$.  Such a bracket will be denoted as $[l, u]$.  A minimal $\epsilon$-covering
of $\mathcal{F}$ with brackets is a collection of brackets such that each element of $\mathcal{F}$ is in
at least one bracket and there exists no collection of $\epsilon$ brackets of smaller cardinality.  The minimum number of brackets
required to cover $\mathcal{F}$ is denoted by $N_{[]}(\epsilon, \mathcal{F})$. 

By Jensen's inequality, we have
$\epsilon^{2} \geq E[|u - v|^{2}] \geq (E[|u - l|])^{2}$ which implies that $E[|u - l|] \leq \epsilon$.

Let $f_{1},f_{2} \in [l, u]$.  We show that $f_{1}$ and $f_{2}$ have variances that are close to one another for small values of $\epsilon$.
This result is used in Theorem \ref{thm::main_extn}.
\begin{lemma}
\label{lemma::var_bound}
Let $f_{1}$ and $f_{2}$ be elements of an $L_{2}$ $\epsilon$-bracket, $[l, u]$.  Then 
\begin{equation}
|\sigma^{2}_{f_{1}} - \sigma^{2}_{f_{2}}|  \leq \epsilon (\sigma_{f_{1}} + \sigma_{f_{2}})
 \end{equation}
\end{lemma}
\begin{proof}
\begin{flalign*}
&E\bigg\{ \big[(f_{1}(O) - E[f_{1}(O)]) - (f_{2}(O) - E[f_{2}(O)])\big]^{2}\bigg\}&\\
&=E\bigg\{\big[f_{1}(O) - f_{2}(O) - E[f_{1}(O)- f_{2}(O)] \big]^{2}\bigg\}&\\
&\leq E\bigg\{\big[(f_{1}(O) - f_{2}(O)) - (E[f_{1}(O)- f_{2}(O)]) \big]^{2}\bigg\} + (E[f_{1}(O) - f_{2}(O)])^{2}&\\
&= E[(f_{1}(O) - f_{2}(O))^{2}] \leq \epsilon^{2}.
\end{flalign*}

This implies that 
\begin{equation*}
\sqrt{E\bigg\{ \big[(f_{1}(O) - E[f_{1}(O)]) - (f_{2}(O) - E[f_{2}(O)])\big]^{2}\bigg\}} \leq \epsilon.
\end{equation*}
By the reverse triangle inequality, we then have 
\begin{flalign*}
&\bigg| \sqrt{E[(f_{1}(O) - E[f_{1}(O)])^{2}]} - \sqrt{E[(f_{2}(O) - E[f_{2}(O)])^{2}]} \bigg|&\\
&= \bigg| \sigma_{f_{1}} - \sigma_{f_{2}} \bigg| \leq \epsilon,
\end{flalign*}
where $\sigma^{2}_{f_{i}}$ is the variance of $f_{i} (i = 1, 2)$. 

Then $|\sigma^{2}_{f_{1}}$ - $\sigma^{2}_{f_{2}}| = |(\sigma_{f_{1}} - \sigma_{f_{2}})(\sigma_{f_{1}} + \sigma_{f_{2}})| \leq \epsilon (\sigma_{f_{1}} + \sigma_{f_{2}})$.
\end{proof}

\section{Appendix B: Proofs for Section 2}

\subsection{Proof of Theorem \ref{thm::main_extn} }
\begin{proof}
To simplify notation, for this finite sample result, we suppress the dependence of $\hat{k}_{n}$ and $\tilde{k}_{n}$ on the sample size, thus,
let $\hat{k}_{n}=\hat{k}$ and $\tilde{k}_{n}=\tilde{k}$.
We begin with the same decomposition as  \cite{dudoit2005asymptotics}.  
\begin{flalign*}
0 &\leq \tilde{\theta}^{CV}_{n(1-\pi)}(\hat{k}) - \theta_{opt} \\
&= E_{S_{n}}\int L(y, \psi_{\hat{k}}(x|P^{0}_{n,S_{n}})) - L(y, \psi(x))dP_{0}(x, y)\\
&\-\hspace{.5cm}-  (1 + \delta)E_{S_{n}}\int L(y, \psi_{\hat{k}}(x|P^{0}_{n,S_{n}})) - L(y, \psi(x))dP^{1}_{n,S_{n}}(x,y)\\
&\-\hspace{.5cm}+ (1 + \delta)E_{S_{n}}\int L(y, \psi_{\hat{k}}(x|P^{0}_{n,S_{n}})) - L(y, \psi(x))dP^{1}_{n,S_{n}}(x,y)\\
&\leq  E_{S_{n}}\int L(y, \psi_{\hat{k}}(x|P^{0}_{n,S_{n}})) - L(y, \psi(x))dP_{0}(x, y)\\
&\-\hspace{.5cm}-  (1 + \delta)E_{S_{n}}\int L(y, \psi_{\hat{k}}(x|P^{0}_{n,S_{n}})) - L(y, \psi(x))dP^{1}_{n,S_{n}}(x,y)\\
&\-\hspace{.5cm}+ (1 + \delta)E_{S_{n}}\int L(y, \psi_{\tilde{k}}(x|P^{0}_{n,S_{n}})) - L(y, \psi(x))dP^{1}_{n,S_{n}}(x,y)\\
&\-\hspace{.5cm}+(1 + \delta)\gamma\\
&=E_{S_{n}}\int L(y, \psi_{\hat{k}}(x|P^{0}_{n,S_{n}})) - L(y, \psi(x))dP_{0}(x, y)\\
&\-\hspace{.5cm}-  (1 + \delta)E_{S_{n}}\int L(y, \psi_{\hat{k}}(x|P^{0}_{n,S_{n}})) - L(y, \psi(x))dP^{1}_{n,S_{n}}(x,y)\\
&\-\hspace{.5cm}+ (1 + \delta)E_{S_{n}}\int L(y, \psi_{\tilde{k}}(x|P^{0}_{n,S_{n}})) - L(y, \psi(x))dP^{1}_{n, S_{n}}(x, y)\\
&\-\hspace{.5cm} - (1 + 2\delta)E_{S_{n}}\int L(y, \psi_{\tilde{k}}(x|P^{0}_{n,S_{n}})) - L(y, \psi(x))dP_{0}(x, y)\\
&\-\hspace{.5cm}+ (1 + 2\delta)E_{S_{n}}\int L(y, \psi_{\tilde{k}}(x|P^{0}_{n,S_{n}})) - L(y, \psi(x))dP_{0}(x, y)\\
&\-\hspace{.5cm}+(1 + \delta)\gamma
\end{flalign*}
As in \cite{dudoit2005asymptotics}, denote the sum of the first and second terms above by $R_{n\hat{k}}$.  Denote the sum of the third
and fourth term as $T_{n,\tilde{k}}$.  The fifth term is the cross-validation benchmark.  Therefore, we have
\begin{equation}
0 \leq \tilde{\theta}^{CV}_{n(1-\pi)}(\hat{k}) - \theta_{opt} \leq (1 + 2\delta)\{\tilde{\theta}^{CV}_{n(1-\pi)}(\tilde{k}) - \theta_{opt}\} + R_{n\hat{k}} + T_{n,\tilde{k}} + (1 + \delta)\gamma.
\end{equation}
The objective is then to find upper bounds for $ER_{n\hat{k}}$ and $ET_{n,\tilde{k}}$.
We will show that the same bound applies for both $ER_{n\hat{k}}$ and $ET_{n,\tilde{k}}$. 
At present, we find an upper bound for $R_{n\hat{k}}$.

Fixing our attention on a particular split into training and validation sets, say the $j$th split $\splitj$, let $Z_{ki}^{S_{n}^{(j)}} = Z_{ki}^{\splitj}(Y_{i},X_{i})= L(Y_{i}, \psi_{k}(X_{i}|P_{n,\splitj}^{0})) - L(Y_{i}, \psi(X_{i}))$, where $i$ is such that $S^{(j)}_{ni}=1$.
Let $Z_{k}^{\splitj} = Z_{k}^{\splitj}(Y, X) = L(Y, \psi_{k}(X|P_{n,\splitj}^{0})) - L(Y, \psi(X))$ have the same distribution as $Z^{\splitj}_{ki}$. 
Note that $|Z^{\splitj}_{ki}|\leq M_{1}$ $a.s$.  We also have $\int Z_{k}^{S^{(j)}_{n}}(x, y) dP_{0}(x, y) = E[Z_{k}^{\splitj}|P^{0}_{n, S_{n}^{(j)}}] - \theta_{opt} \geq 0$. 

For any $k$, let 
\begin{equation*}
R_{nk} = \frac{1}{K}\sum_{j=1}^{K}R_{nk}(S^{(j)}_{n}),
\end{equation*}
where
\begin{equation*}
R_{nk}(S^{(j)}_{n})=\int Z_{k}^{S^{(j)}_{n}}(x, y) dP_{0}(x, y) - (1 + \delta)\big\{\frac{1}{n\pi}\sum_{\{i : S_{ni}^{(j)}=1\}}Z_{ki}^{S^{(j)}_{n}}(X_{i}, Y_{i})\big\}.
\end{equation*}
After adding and subtracting $\delta\int Z_{k}^{S^{(j)}_{n}}(x, y) dP_{0}(x, y)$,
we have 
\begin{equation}
R_{nk}(S^{(j)}_{n})=(1 + \delta)\bigg\{\frac{1}{n\pi}\sum_{\{i : S^{(j)}_{ni}=1\}}(\int Z_{k}^{S^{(j)}_{n}}(x, y) dP_{0}(x, y) - Z_{ki})\bigg\} - \delta\int Z_{k}^{S^{(j)}_{n}}(x, y) dP_{0}(x, y).\label{eq::rnkdef}
\end{equation}

Therefore, 
\begin{flalign}
\nonumber & P\bigg(R_{n\hat{k}}(S_{n}^{(j)}) > s | P^{0}_{n,S^{(j)}_{n}}\bigg) &\\
\nonumber &= P\bigg((1 + \delta)\big\{\frac{1}{n\pi}\sum_{\{i : S^{(j)}_{ni}=1\}} E[Z_{\hat{k}i}^{S^{(j)}_{n}}|P^{0}_{n,S^{(j)}_{n}}] - Z_{\hat{k}i}^{S^{(j)}_{n}}\big\} - \delta E[Z_{\hat{k}i}^{S^{(j)}_{n}}|P^{0}_{n,S^{(j)}_{n}}]> s |P^{0}_{n,\splitj}\bigg)&\\
\nonumber &\leq P\bigg( \sup_{k\in \Xi_{n}}\bigg\{ (1 + \delta)\big\{\frac{1}{n\pi}\sum_{\{i : S^{(j)}_{ni}=1\}} E[Z_{ki}^{S^{(j)}_{n}}|P^{0}_{n,S^{(j)}_{n}}] - Z_{ki}^{S^{(j)}_{n}}\big\} - \delta E[Z_{ki}^{S^{(j)}_{n}}|P^{0}_{n,S^{(j)}_{n}}] \bigg\} > s|P^{0}_{n,S^{(j)}_{n}}\bigg)&\\
\nonumber &\leq P\bigg( \sup_{k\in \Xi_{n}}\bigg\{ (1 + \delta)\big\{\frac{1}{n\pi}\sum_{\{i : S^{(j)}_{ni}=1\}} E[Z_{ki}^{S^{(j)}_{n}}|P^{0}_{n,S^{(j)}_{n}}] - Z_{ki}^{S^{(j)}_{n}}\big\} - \delta \frac{Var(Z_{ki}^{S^{(j)}_{n}}|P^{0}_{n,S^{(j)}_{n}})}{M_{2}} \bigg\} > s|P^{0}_{n,S^{(j)}_{n}}\bigg),    
\end{flalign}
where in the last inequality we have used the inequality
\begin{equation*}
Var(Z_{ki}^{S^{(j)}_{n}}|P^{0}_{n,S^{(j)}_{n}})/M_{2} \leq E[Z_{ki}^{S^{(j)}_{n}}|P^{0}_{n,S^{(j)}_{n}}]
\end{equation*}
(see Lemma 3 of \cite{dudoit2005asymptotics}).
We will find an exponential bound for the previous probability via bracketing numbers.  

Let $\mathcal{F}_{P^{0}_{n,S_{n}^{(j)}}}=\{Z_{k}^{S^{(j)}_{n}}=Z_{k}^{S^{(j)}_{n}}(X,Y): k\in \Xi_{n} \}$.  Set $\epsilon = s/(4(1 + 2\delta))$.
Let $[l_{v}, u_{v}]$ $(v=1,\ldots,N(\epsilon))$ be a minimal $L_{2}$ $\epsilon$-bracketing of $\mathcal{F}_{P^{0}_{n,S_{n}^{(j)}}}$ where 
$N(\epsilon)=N_{[]}(\epsilon, \mathcal{F}_{P^{0}_{n,\splitj}})$.  By $A1$ we may assume without loss of generality that $|l_{v}|\leq M_{1}$
and that $|u_{v}|\leq M_{1}$.

Note that $N_{[]}(\epsilon, \mathcal{F}_{P^{0}_{n, \splitj}})$ depends on the training set defined by $\splitj$.
Given the training set used in split $\splitj$, $N(\epsilon)=N_{[]}(\epsilon, \mathcal{F}_{P^{0}_{n,\splitj}})$ is a fixed number rather than a random variable.  
We will find a bound for $N(\epsilon, \mathcal{F}_{P^{0}_{n,\splitj}})$ that is independent of the training set.
Choose a representative, $f_{v} \in [l_{v}, u_{v}]$ from each bracket and let $f_{vi}=f_{v}(O_{i})$.  
Similarly, let $l_{vi}=l_{v}(O_{i})$ and $u_{vi}=u_{v}(O_{i})$.
If $Z_{k}^{S_{n}^{(j)}} \in \mathcal{F}_{P^{0}_{n,S_{n}^{(j)}}}$ such that $Z_{k}^{S_{n}^{(j)}} \in [l_{v}, u_{v}]$ we have
\begin{flalign}
\nonumber & \frac{1}{n\pi}\sum_{\{i : S_{ni}^{(j)}=1\}}(Z_{ki}^{\splitj} - E[Z_{ki}^{\splitj}|P^{0}_{n,S_{n}^{(j)}}])&\\
&\leq  \frac{1}{n\pi}|\sum_{\{i : S_{ni}^{(j)}=1\}}f_{vi} - E[f_{vi}|P^{0}_{n,S_{n}^{(j)}}]| +&\\ 
&\-\hspace{.5cm}\frac{1}{n\pi}|\sum_{\{i : S_{ni}^{(j)}=1\}}(Z_{ki}^{\splitj} - f_{vi}) - E[(Z_{ki}^{\splitj} - f_{vi})|P^{0}_{n,S_{n}^{(j)}}]|. \label{fixk:line:ineq}
\end{flalign}

Consider the term in (\ref{fixk:line:ineq}).  By the triangle inequality and the fact that
$[l_{v}, u_{v}]$ is an $\epsilon$ bracket we have
\begin{flalign*}
&\frac{1}{n\pi}|\sum_{\{i : S_{ni}^{(j)}=1\}}(Z_{ki}^{\splitj} - f_{vi}) - E[(Z_{ki}^{\splitj} - f_{vi})|P^{0}_{n,S_{n}^{(j)}}]| &\\
&\leq \frac{1}{n\pi}\sum_{\{i : S_{ni}^{(j)}=1\}}|Z_{ki}^{\splitj} - f_{vi}| + E[|Z_{ki}^{\splitj} - f_{vi}||P^{0}_{n,S_{n}^{(j)}}] \\
&\leq \frac{1}{n\pi}\sum_{\{i : S_{ni}^{(j)}=1\}}|u_{vi} - l_{vi}| + E[|u_{vi} - l_{vi}||P^{0}_{n,S_{n}^{(j)}}] \\
&= \frac{1}{n\pi}\sum_{\{i : S_{ni}^{(j)}=1\}}|u_{vi} - l_{vi}| - E[|u_{vi} - l_{vi}||P^{0}_{n,S_{n}^{(j)}}]| + 2E[|u_{vi} - l_{vi}||P^{0}_{n,S_{n}^{(j)}}]\\
&\leq \frac{1}{n\pi}(\sum_{\{i : S_{ni}^{(j)}=1\}}|u_{vi} - l_{vi}| - E[|u_{vi} - l_{vi}||P^{0}_{n,S_{n}^{(j)}}]|) + 2\epsilon.
\end{flalign*}

Replacing the term in (\ref{fixk:line:ineq}) by the final term in the above series of inequalities 
yields the inequality
\begin{flalign}
\nonumber &\frac{1}{n\pi}|\sum_{\{i : S_{ni}^{(j)}=1\}}(Z_{ki}^{\splitj} - E[Z_{ki}^{\splitj}|P^{0}_{n,S_{n}^{(j)}}])|\\
 &\leq  \frac{1}{n\pi}|\sum_{\{i : S_{ni}^{(j)}=1\}}f_{vi} - E[f_{vi}|P^{0}_{n,S_{n}^{(j)}}]| +& \label{eq::bracketinq1}\\ 
 &\-\hspace{.5cm} \frac{1}{n\pi}\sum_{\{i : S_{ni}^{(j)}=1\}}(|u_{vi} - l_{vi}| - E[|u_{vi} - l_{vi}||P^{0}_{n,S_{n}^{(j)}}]|) + 2\epsilon \label{eq::bracketinq2}.
\end{flalign} 

Using Lemma \ref{lemma::var_bound} and the fact that $\sqrt{Var(Z_{ki}^{\splitj}|P^{0}_{n,S_{n}^{(j)}})} \leq M_{2}$ 
and $\sqrt{Var(f_{v}|P^{0}_{n,S_{n}^{(j)}})} \leq M_{2}$ , we have 
\begin{equation}
Var(Z_{ki}^{\splitj}|P^{0}_{n,S_{n}^{(j)}}) \geq Var(f_{v}|P^{0}_{n,S_{n}^{(j)}}) - 2M_{2}\epsilon. \label{eq::var_inq}
\end{equation}

Combining (\ref{eq::bracketinq1}), (\ref{eq::bracketinq2}), and (\ref{eq::var_inq}), we have 
\begin{flalign*}
&(1 + \delta)\big\{\frac{1}{n\pi}\sum_{\{i : S^{(j)}_{ni}=1\}} E[Z_{ki}^{\splitj}|P^{0}_{n,S^{(j)}_{n}}] - Z_{ki}^{\splitj}\big\} - \delta \frac{Var(Z_{ki}^{S^{(j)}}|P^{0}_{n,S^{(j)}_{n}})}{M_{2}}\\
&\leq (1 + \delta)\frac{1}{n\pi}\sum_{\{i : S_{ni}^{(j)}=1\}}(|u_{vi} - l_{vi}| - E[|u_{vi} - l_{vi}||P^{0}_{n,S_{n}^{(j)}}]|)&\\
&+ (1 + \delta)\frac{1}{n\pi}|\sum_{\{i : S_{ni}^{(j)}=1\}}f_{vi} - E[f_{vi}|P^{0}_{n,S_{n}^{(j)}}]| - \delta \frac{Var(f_{v}|P^{0}_{n,S_{n}^{(j)}})}{M_{2}}&\\
&+ 2(1+\delta)\epsilon +  2\delta\epsilon. 
\end{flalign*}

Making use of the fact that $s - 2(1 + \delta)\epsilon - 2\delta\epsilon = s/2$, we have
\begin{flalign*}
&P\bigg( \sup_{k\in \Xi_{n}}\bigg\{ (1 + \delta)\big\{\frac{1}{n\pi}\sum_{\{i : S^{(j)}_{ni}=1\}} E[Z_{ki}^{\splitj}|P^{0}_{n,S^{(j)}_{n}}] - Z_{ki}^{\splitj}\big\} - \delta \frac{Var(Z_{ki}^{S^{(j)}}|P^{0}_{n,S^{(j)}_{n}})}{M_{2}} \bigg\} > s|P^{0}_{n,S^{(j)}_{n}}\bigg)&\\
&\leq P\bigg( \sup_{v\in \{1,\ldots, N(\epsilon)\}}\bigg\{ (1 + \delta)\frac{1}{n\pi}\sum_{\{i : S_{ni}^{(j)}=1\}}(|u_{vi} - l_{vi}| - E[|u_{vi} - l_{vi}||P^{0}_{n,S_{n}^{(j)}}]|)+&\\
&\-\hspace{3cm}(1 + \delta)\frac{1}{n\pi}|\sum_{\{i : S_{ni}^{(j)}=1\}}f_{vi} - E[f_{vi}|P^{0}_{n,S_{n}^{(j)}}]| - \delta \frac{Var(f_{v}|P^{0}_{n,S_{n}^{(j)}})}{M_{2}}\bigg\} > \frac{s}{2}|P^{0}_{n,S^{(j)}_{n}}\bigg)&\\
&\leq N(\epsilon)\max_{v\in \{1,\ldots, N(\epsilon)\}} P\bigg(\bigg\{ (1 + \delta)\frac{1}{n\pi}\sum_{\{i : S_{ni}^{(j)}=1\}}(|u_{vi} - l_{vi}| - E[|u_{vi} - l_{vi}||P^{0}_{n,S_{n}^{(j)}}]|)+&\\
&\-\hspace{3cm}(1 + \delta)\frac{1}{n\pi}|\sum_{\{i : S_{ni}^{(j)}=1\}}f_{vi} - E[f_{vi}|P^{0}_{n,S_{n}^{(j)}}]| - \delta \frac{Var(f_{v}|P^{0}_{n,S_{n}^{(j)}})}{M_{2}}\bigg\} > \frac{s}{2}|P^{0}_{n,S^{(j)}_{n}}\bigg)&\\
\end{flalign*}

Fixing a particular value of $v$, we use the general inequality that for any pair of random variables $A$ and $B$, $P(A + B > a + b) \leq P(A > a) + P(B > b)$:
\begin{flalign}
\nonumber&P\bigg(\bigg\{ (1 + \delta)\frac{1}{n\pi}\sum_{\{i : S_{ni}^{(j)}=1\}}(|u_{vi} - l_{vi}| - E[|u_{vi} - l_{vi}||P^{0}_{n,S_{n}^{(j)}}]|)+&\\
\nonumber&\-\hspace{1cm}(1 + \delta)\frac{1}{n\pi}|\sum_{\{i : S_{ni}^{(j)}=1\}}f_{vi} - E[f_{vi}|P^{0}_{n,S_{n}^{(j)}}]| - \delta \frac{Var(f_{v}|P^{0}_{n,S_{n}^{(j)}})}{M_{2}}\bigg\} > \frac{s}{2}|P^{0}_{n,S^{(j)}_{n}}\bigg)&\\
&\leq P\bigg((1 + \delta)\frac{1}{n\pi}\sum_{\{i : S_{ni}^{(j)}=1\}}(|u_{vi} - l_{vi}| - E[|u_{vi} - l_{vi}||P^{0}_{n,S_{n}^{(j)}}]|) > \frac{s}{4} | P^{0}_{n,S^{(j)}_{n}}\bigg)\label{eq::bern1}&\\
&+ P\bigg((1 + \delta)\frac{1}{n\pi}|\sum_{\{i : S_{ni}^{(j)}=1\}}f_{vi} - E[f_{vi}|P^{0}_{n,S_{n}^{(j)}}]| - \delta\frac{Var(f_{v}|P^{0}_{n,S_{n}^{(j)}})}{M_{2}}   > \frac{s}{4}|P^{0}_{n,S^{(j)}_{n}}\bigg)\label{eq::bern2}
\end{flalign}

First we obtain an upper-bound for (\ref{eq::bern2}).  For simplicity let $Var(f_{v}|P^{0}_{n,S_{n}^{(j)}})=\sigma_{f_{v}}^{2}$. By Bernstein's inequality we have,
\begin{flalign*}
&P\bigg((1 + \delta)\frac{1}{n\pi}|\sum_{\{i : S_{ni}^{(j)}=1\}}f_{vi} - E[f_{vi}|P^{0}_{n,S_{n}^{(j)}}]| - \delta\frac{\sigma_{f_{v}}^{2}}{M_{2}}   > \frac{s}{4}|P^{0}_{n,S^{(j)}_{n}}\bigg)&\\
&=P\bigg(\frac{1}{n\pi}|\sum_{\{i : S_{ni}^{(j)}=1\}}f_{vi} - E[f_{vi}|P^{0}_{n,S_{n}^{(j)}}]| > \frac{1}{1 + \delta}(\delta\frac{\sigma_{f_{v}}^{2}}{M_{2}} + \frac{s}{4})|P^{0}_{n,S^{(j)}_{n}}\bigg)&\\
&=2\exp\bigg(-\frac{n \pi}{2(1+\delta)^{2}}\frac{(\delta\frac{\sigma_{f_{v}}^{2}}{M_{2}} + \frac{s}{4})^{2}}{\sigma_{f_{v}}^{2} + \frac{M_{1}}{3(1 + \delta)}(\delta\frac{\sigma_{f_{v}}^{2}}{M_{2}} + \frac{s}{4})} \bigg)&
\end{flalign*}

We simplify the following expression in the exponent:  
\begin{flalign*}
&\frac{(\delta\frac{\sigma_{f_{v}}^{2}}{M_{2}} + \frac{s}{4})^{2}}{\sigma_{f_{v}}^{2} + \frac{M_{1}}{3(1 + \delta)}(\delta\frac{\sigma_{f_{v}}^{2}}{M_{2}} + \frac{s}{4}) } = 
\frac{(\delta\frac{\sigma_{f_{v}}^{2}}{M_{2}} + \frac{s}{4})}{\frac{\sigma_{f_{v}}^{2}}{(\delta\frac{\sigma_{f_{v}}^{2}}{M_{2}} + \frac{s}{4})} + \frac{M_{1}}{3(1 + \delta)}} 
\geq \frac{(\delta\frac{\sigma_{f_{v}}^{2}}{M_{2}} + \frac{s}{4})}{\frac{M_{2}}{\delta} + \frac{M_{1}}{3(1 + \delta)}} 
\geq \frac{1}{4}\frac{s}{\frac{M_{2}}{\delta} + \frac{M_{1}}{3}} &
\end{flalign*}

As $c_{2}(M, \delta) = (1 + \delta)^{2}8(\frac{M_{2}}{\delta} + \frac{M_{1}}{3})$, we therefore have
\begin{flalign}
\nonumber &P\bigg((1 + \delta)\frac{1}{n\pi}|\sum_{\{i : S_{ni}^{(j)}=1\}}f_{vi} - E[f_{vi}|P^{0}_{n,S_{n}^{(j)}}]| - \delta\frac{\sigma_{f_{v}}^{2}}{M_{2}}   > \frac{s}{4}|P^{0}_{n,S^{(j)}_{n}}\bigg)&\\
&\leq 2\exp\bigg(-\frac{n\pi}{c_{2}(M,\delta)} s \bigg) \label{inq:domterm}.
\end{flalign}

Now consider the term in (\ref{eq::bern1}).  Using Bernstein's inequality and the fact that $Var(|u_{vi} - l_{vi}|) \leq E|u_{vi} - l_{vi}|^{2} \leq E|u_{vi} - l_{vi}|M_{1} \leq M_{1}\epsilon$,
we have 
\begin{flalign}
\nonumber &P(\frac{1}{n\pi}(\sum_{\{i : S^{(j)}_{ni}=1\}}|u_{vi} - l_{vi}| - E[|u_{vi} - l_{vi}||P^{0}_{n,S^{(j)}_{n}},S^{(j)}_{n}])  > \frac{s}{4} | P^{0}_{n,\splitj},\splitj)\\
\nonumber &\leq 2\exp(-\frac{1}{2}\frac{n\pi(\frac{s}{4})^{2}}{M_{1}\epsilon+ M_{1}\frac{s}{4}\frac{1}{3}})\\ 
\nonumber &= 2\exp(-\frac{1}{2}\frac{n\pi(\frac{s}{4})^{2}}{M_{1}\frac{s}{4(1 + 2\delta)}+M_{1}\frac{s}{4}\frac{1}{3}})\\
\nonumber &= 2\exp(-\frac{1}{M_{1}8}\frac{n\pi}{\frac{1}{(1 + 2\delta)}+\frac{1}{3}}s). \label{inq:weakterm}\\ 
\end{flalign}

By $A4$, $\delta$ is taken small enough so that $1/(c_{2}(M,\delta)) \leq M_{1}8(\frac{1}{(1 + 2\delta)}+\frac{1}{3})$ and, in this case, the upper bound for (\ref{eq::bern2}) (Inequality \ref{inq:domterm}) is
larger than the upper bound for (\ref{eq::bern1}) (Inequality \ref{inq:weakterm}) and we have 
\begin{flalign*}
&P(R_{n,\hat{k}}(\splitj) > s|P^{0}_{n,\splitj})&\\
&\leq N(\epsilon)4\exp(-\frac{n\pi}{c_{2}(M,\delta)}s)\\
&\leq N(\frac{s}{4(1+2\delta)})4\exp(-\frac{n\pi}{c_{2}(M,\delta)}s).
\end{flalign*}

By assumption $\mathcal{F}_{P^{0}_{n,\splitj}}=\{Z_{k}: k\in \Xi_{n} \}$ are Lipshitz, with Lipshitz constant $C$
independent of $P^{0}_{n,\splitj}$.  Therefore, by Example 19.7 in \cite{van2000asymptotic} and Example 27.1 of \cite{shalev2014understanding},
\begin{flalign*}
N(\frac{s}{4(1+2\delta)}) \leq \Big(\frac{4\sqrt{d}C^{2}4(1+2\delta)diam(\Xi_{n})}{s}\Big)^{d}. 
\end{flalign*}

Thus, 
\begin{flalign*}
P(R_{n,\hat{k}}(\splitj) > s|P^{0}_{n,\splitj}) \leq 4\Big(\frac{4\sqrt{d}C^{2}4(1+ 2\delta)\text{diam}(\Xi_{n})}{s}\Big)^{d}\exp(-\frac{n\pi}{c_{2}(M,\delta)}s)
\end{flalign*}
and this implies that 
\begin{flalign*}
P(R_{n,\hat{k}}(\splitj) > s) \leq 4\Big(\frac{4\sqrt{d}C^{2}4(1+ 2\delta)\text{diam}(\Xi_{n})}{s}\Big)^{d}\exp(-\frac{n\pi}{c_{2}(M,\delta)}s)
\end{flalign*}

Note that $ER_{n, \hat{k}}= ER_{n, \hat{k}}(\splitj)$.
In general, for any random variable $Z$, $EZ \leq EI(Z > 0)Z = \int_{0}^{\infty}P(Z > z)dz \leq 
u + \int_{u}^{\infty}P(Z > z)dz$ $(u > 0)$.
If we assume $u > (\frac{c_{2}(M,\delta)}{n\pi})$, we obtain the following upper bound:
\begin{flalign*}
&ER_{n,\hat{k}} = ER_{n, \hat{k}}(\splitj) \leq u + \int_{u}^{\infty}4\Big(\frac{4\sqrt{d}C^{2}4(1+ 2\delta)\text{diam}(\Xi_{n})}{s}\Big)^{d}\exp(-\frac{n\pi}{c_{2}(M,\delta)}s)ds&\\
&\leq u + \Big(\frac{n\pi}{c_{2}(M, \delta)}\Big)^{d}4(4\sqrt{d}C^{2}4(1+ 2\delta)\text{diam}(\Xi_{n}))^{d}\int_{u}^{\infty}\exp(-\frac{n\pi}{c_{2}(M, \delta)}s)ds\\
&=u + \Big(\frac{n\pi}{c_{2}(M, \delta)}\Big)^{d}4(4\sqrt{d}C^{2}4(1+ 2\delta)\text{diam}(\Xi_{n}))^{d}\frac{c_{2}(M, \delta)}{n\pi}\exp(-\frac{n\pi}{c_{2}(M, \delta)}u)
\end{flalign*}
If we let 
\begin{flalign*}
u=\frac{c_{2}(M, \delta)}{n\pi}\log\Big\{ \Big(\frac{n\pi}{c_{2}(M, \delta)}\Big)^{d}4(4\sqrt{d}C^{2}4(1+ 2\delta)\text{diam}(\Xi_{n}))^{d} \Big\},
\end{flalign*}
we have 
\begin{flalign*}
&ER_{n,\hat{k}} \leq \frac{c_{2}(M, \delta)}{n\pi}c_{1}(n\pi, d, \Xi_{n}, M, \delta) + \frac{c_{2}(M, \delta)}{n\pi}&\\
&\hspace{1.15cm}\leq 2\frac{c_{2}(M, \delta)}{n\pi}c_{1}(n\pi, d, \Xi_{n}, M, \delta)
\end{flalign*}
(by $A4$, $u$ is indeed larger than $(c_{2}(M, \delta)/n\pi$)).  The last inequality also follows by $A4$ as $c_{1}(n\pi, d, \Xi_{n}, M, \delta)$ is assumed to be larger than 1.

The same bound also holds for $ET_{n,\tilde{k}}$ as $T_{n, \tilde{k}}$ has the same form as $R_{n, \hat{k}}$. This is seen by noting that 
\begin{flalign*}
&T_{n, \tilde{k}}=(1 + \delta)E_{S_{n}}\int L(y, \psi_{\tilde{k}}(x|P^{0}_{n,S_{n}})) - L(y, \psi(x))dP^{1}_{n, S_{n}}(x, y)\\
&\hspace{1.25cm}- (1 + 2\delta)E_{S_{n}}\int L(y, \psi_{\tilde{k}}(x|P^{0}_{n,S_{n}})) - L(y, \psi(x))dP_{0}(x, y)\\
&\hspace{.75cm}= (1 + \delta)E_{S_{n}}\int L(y, \psi_{\tilde{k}}(x|P^{0}_{n,S_{n}})) - L(y, \psi(x))dP^{1}_{n, S_{n}}(x, y)\\
&\hspace{1.25cm}-(1 + \delta)E_{S_{n}}\int L(y, \psi_{\tilde{k}}(x|P^{0}_{n,S_{n}})) - L(y, \psi(x))dP_{0}(x, y)\\
&\hspace{1.25cm}-\delta E_{S_{n}}\int L(y, \psi_{\tilde{k}}(x|P^{0}_{n,S_{n}})) - L(y, \psi(x))dP_{0}(x, y)
\end{flalign*}
Noting that $Z_{\tilde{k}}^{\splitj}(x, y) = L(y, \psi_{\tilde{k}}(x|P^{0}_{n,\splitj})) - L(y, \psi(x))$, and comparing the above expression
with that in (\ref{eq::rnkdef}), we see that the upper bound obtained for $ER_{n, \hat{k}}$ also holds for $ET_{n,\tilde{k}}$.

Thus, we have 
\begin{flalign*}
&0 \leq E\tilde{\theta}^{CV}_{n(1-\pi)}(\hat{k}) - \theta_{opt} \leq (1+2\delta)\{E\tilde{\theta}^{CV}_{n(1-p)}(\tilde{k}) - \theta_{opt}\}  + &\\
&\-\hspace{4.5cm} \frac{4c_{2}(M,\delta)}{n\pi}c_{1}(n\pi, d, \Xi_{n}, M, \delta) +&\\
&\-\hspace{4.5cm} (1 + \delta)\gamma                                       
\end{flalign*}

\end{proof}

\subsection{Proof of Theorem \ref{thm::kern_main}}
\begin{proof}
The $K$-fold cross-validation criterion is continuous as a function of $H$ (the choice of Gaussian kernel is important for this
purpose).  Note that $\Xi_{n}$ is closed and bounded, therefore $\Xi_{n}$ is compact. 
As $\Xi_{n}$ is compact and the cross-validation criterion is continuous, there exists a minimizer of the $K$-fold
cross-validation criterion in $\Xi_{n}$.  The cross-validation benchmark for $K$-fold cross-validation is proportional to the sum
of conditional risks for $K$ kernel regression estimators each based off of
a training set of size $n(1-\pi)$.  If we can show that one of these conditional risks is continuous as a function of $k$, it will follow
that the cross-validation benchmark is continuous.  Fixing, say, the $j$th split, $ \int L(y, \psi_{H}(x|P^{0}_{n,S_{n}^{(j)}})dP_{0}(x,y)$,
is continuous as a consequence of the dominated convergence theorem (see Theorem 6.27 of \cite{klenke2013probability}),
therefore the cross-validation benchmark is continuous.  Again, because $\tilde{\theta}^{CV}_{n(1-\pi)}(H)$ is continuous 
and $\Xi_{n}$ is compact, there also exists a minimizer of $\tilde{\theta}^{CV}_{n(1-\pi)}(H)$  in $\Xi_{n}$.
Thus, part ($a$) of Theorem \ref{thm::kern_main} holds.

$A1$ holds because the kernel regression estimator is bounded by $M$.
$A2$ is satisfied because $\Xi_{n}$ is bounded.  $A4$ holds by
assumption.  This leaves us to show that $A3$ holds.

Let $k_{i,H}=K(H^{1/2}(X_i-x))$.
Consider the partial derivative of $L(y, \psi_{H}(x|P_{n})) - L(y, \psi(x))$, where $x=(x_{1},\ldots,x_{p})'$ is in the
support of $X$.
We have for $u \neq v$,
\begin{flalign*}
&\bigg| \pd[L(y, \psi_{H}(x|P_{n})) - L(y, \psi(x))] \bigg|&\\ 
&= \bigg| 4(\psi_{H}(x|P_{n}) - y)\frac{\sum_{i=1}^{n} (\psi_{H}(x|P_{n}) - Y_{i})k_{i,H}(x)(X_{iu} - x_{u})(X_{iv} - x_{v})}{\sum_{i=1}^{n} k_{i,H}(x)} \bigg|\\
&\leq 64M\sum_{i=1}^{n}\frac{k_{i,H}(x)MB^{2}}{\sum_{i=1}^{n} k_{i,H}(x)}=64M^{2}B^{2},
\end{flalign*}
where, for the inequality, we have used the fact that $|(\psi_{H}(x|P_{n}) - y)| \leq 2M$,  $|(X_{iu} - x_{u})| \leq 2B$,
and $|(X_{iv} - x_{v})| \leq 2B$.

When $u = v$ we have the same expression as above with the 4 replaced by 2.  Thus, the above derivative is
bounded above by a constant and the constant is independent of the training set.  As a function of $H$,
$L(y, \psi_{H}(x|P_{n})) - L(y, \psi(x))$ is continuously differentiable and has derivative bounded by 
$64M^{2}B^{2}$, therefore by \cite{duistermaat2004multidimensional} (Exercise 2.15) it is Lipshitz continuous with Lipshitz constant
$C=64\sqrt{p(p + 1)/2}B^{2}M^{2}$.  Thus, condition $A3$ is satisfied and part ($b$) is proven.

Finally, we prove part ($c$).
Let $\tilde{H}_{n}$ be a bandwidth matrix that minimizes the cross-validation benchmark $\tilde{\theta}_{n(1 - \pi)}^{CV}(H)$.
By definition, for any other fixed bandwidth matrix, $\breve{H}_{n} \in \Xi_{n}$, the below inequality holds:
\begin{equation*}
E\tilde{\theta}_{n(1-\pi)}^{CV}(\tilde{H}_{n}) - \theta_{opt} \leq E\tilde{\theta}_{n(1-\pi)}^{CV}(\breve{H}_{n}) - \theta_{opt}.
\end{equation*}
In addition, because the cross-validation technique being utilized is $K$-fold cross-validation and by the properties of
conditional risk discussed in the introduction, it is the case that for $\breve{H}_{n}$ we have, 
\begin{equation*}
E\tilde{\theta}_{n(1-\pi)}^{CV}(\breve{H}_{n}) -\theta_{opt} = E_{O_{1},\ldots,O_{n(1-\pi)},X}[(\psi_{n(1-\pi), \breve{H}_{n}}(X) - \psi(X))^{2}]
\end{equation*}
where the expectation is taken over a training set of size $n(1-\pi)$, $O_{1},\ldots,O_{n(1-\pi)} \sim P_{0, X}$, and a newly observed
covariate, $X \sim P_{0, X}$.  
The conditions of Theorem \ref{thm::kern_main} will be met for sufficiently small $\delta$ and for sufficiently large $n$.
In this case, we have the following upper bound for $E\tilde{\theta}_{n(1-\pi)}^{CV}(\hat{H}_{n}) - \theta_{opt}$:
\begin{flalign}
 \label{eq::breveHinq}
 &E\tilde{\theta}^{CV}_{n(1-\pi)}(\hat{H}_{n}) - \theta_{opt} \leq (1+2\delta)(E\tilde{\theta}^{CV}_{n(1-\pi)}(\breve{H}_{n}) - \theta_{opt})+&\\  
&\nonumber \-\hspace{4.5cm}\frac{4c_{2}(M,\delta)}{n\pi}c_{1}(n\pi, d, \Xi_{n}, M, \delta).
\end{flalign}

Given our choice of $\lambda_{n}$, we will show there exists a sequence of bandwidth matrices $\breve{H}_{n} \in \Xi_{n}$
that yields the desired result.  Furthermore, $\lambda_{n}$ must go
to $\infty$ at a rate slow enough such that the second term on the right-hand side of the inequality in (\ref{eq::breveHinq}) is of smaller order than the first term.  

Define 
\begin{equation*}
h_{n}(q)=V(\log(n)n)^{1/(q + 2)},
\end{equation*}
where $V > 0$ is positive constant.
Note that 
\begin{equation*}
\lambda_{n}=\sqrt{p}V(\log(n)n)^{1/3} = \sqrt{p}\max_{q \in \{{1, \ldots, p}\}}h_{n}(q) 
\end{equation*}  
Choose $\breve{H}_{n} = T'\Lambda T$, where $\Lambda$ is a diagonal matrix with its $m$ diagonal entries equal to $h_{n}(m)$.
Let $||\breve{H}_{n}||_{O}$ be the operator norm or the largest eigenvalue of the matrix $\breve{H}_{n}$.
As the largest eigenvalue of $\breve{H}_{n}$ is $h_{n}(m)$, $||\breve{H}_{n}||_{O} =h_{n}(m)$.
We have $||\breve{H}_{n}||_{F} \leq \sqrt{p}||\breve{H}_{n}||_{O} = \sqrt{p}h_{n}(m) \leq \lambda_{n}$.  
Therefore, $\breve{H}_{n} \in \Xi_{n}$.   
 
 The kernel regression estimator $\psi_{n(1-\pi), \breve{H}_{n}}(X)$ is equivalent to 
 the kernel regression estimator obtained by regressing the $m$-dimensional covariate vector $TX_{i}$ on $Y_{i}$, using the bandwidth matrix $h_{n}(m)I_{m}$.  
 Denote this equivalent estimator as $\phi_{n(1-\pi), h_{n}(m)}(TX)$.
 We have 
 \begin{flalign*}
 &E\tilde{\theta}_{n(1-\pi)}^{CV}(\breve{H}) -\theta_{opt} = E_{O_{1},\ldots,O_{n(1-\pi)},X}[(\psi_{n(1-\pi), \breve{H}_{n}}(X) - \psi(X))^{2}]&\\
 &\hspace{3.35cm}= E_{O_{1},\ldots,O_{n(1-\pi)},X}[(\phi_{n(1-\pi), h_{n}(m)}(TX) - \phi(TX))^{2}].
 \end{flalign*}
 Thus, by using Inequality (\ref{inq::kernel_bound1}) from
 Lemma \ref{lemma::kern_lemma}:
 \begin{equation}
 E\tilde{\theta}^{CV}_{n(1-\pi)}(\breve{H}) - \theta_{opt} \leq A'\log(n(1 - \pi))^{\frac{m}{m + 2}}(n(1 - \pi))^{\frac{-2}{m + 2}}, \label{inq::kern_upperbnd}
 \end{equation}
 where $A'=(A_{1}V^{-1} + A_{2}V^{\frac{m}{m + 2}})$.  Note that in the application of Lemma \ref{lemma::kern_lemma}, there exists $\sigma^{2}$ such that
 $Var(Y|X=x) \leq \sigma^{2}$ because $Y$ is a.s. bounded.
 
 Next consider the rate at which $diam(\Xi_{n})$ grows.  If $H \in \Xi_{n}$ then $||H|| \leq ||H||_{F} \leq \lambda_{n}$ ($||H||$ is
 simply the Euclidean norm of the unique elements in $H$).    
 Thus, $diam(\Xi_{n})$ grows at rate $O(\lambda_{n})$.  Therefore, the second term on the right-hand side of the inequality (\ref{eq::breveHinq})
 is indeed of smaller order than the first.   
 
 By the properties of the conditional risk discussed in the introduction,
 \begin{equation}
 E\tilde{\theta}_{n(1-\pi)}^{CV}(\hat{H}_{n}) - \theta_{opt} =  E\int (\psi_{n(1-\pi), \hat{H}_{n}}(x|P^{0}_{n, S_{n}^{(j)}}) - \psi(x))^{2}dP_{0,X}(x).
 \end{equation} 
 The desired result then follows by (\ref{eq::breveHinq}). 

\end{proof}

\begin{supplement}[id=suppA]
  \stitle{Supplementary materials for An Oracle Property of the Nadaraya-Watson Kernel Estimator for High Dimensional Nonparametric Regression}
  \slink[doi]{COMPLETED BY THE TYPESETTER}
  \sdatatype{.pdf}
  \sdescription{We present the proof of Lemma 6.1 in this supplementary article as well as the derivative of $K$-fold cross-validation criterion
  with respect to $H$.}
\end{supplement}

\newpage
\bibliographystyle{imsart-nameyear}
\bibliography{MetricLearning2,CVrefs2}

 \end{document}